\documentclass[11pt,a4paper]{article}
\usepackage{fancyhdr}
\usepackage[british]{babel}

\usepackage{amssymb, amsmath, amsthm, amsfonts, enumerate}
\usepackage{amsmath,setspace,scalefnt}
\usepackage[usenames,dvipsnames,svgnames,table]{xcolor}
\usepackage{graphicx,tikz,caption,subcaption}
\usetikzlibrary{patterns}
\usetikzlibrary{shapes}
\usepackage{fullpage}
\usepackage{hyperref}
\usepackage{enumitem,verbatim}
\usepackage{multicol}
\usepackage{float}
\usepackage{cleveref}

\usepackage[margin=2.5cm]{geometry}

\definecolor{gray}{rgb}{0.25, 0.25, 0.25}

\newtheorem{theorem}{Theorem}[section]
\newtheorem{lemma}[theorem]{Lemma}
\newtheorem{cor}[theorem]{Corollary}
\newtheorem{conj}[theorem]{Conjecture}

\newtheorem*{theorem*}{Theorem}

\theoremstyle{definition}

\theoremstyle{plain}

\newtheorem{prop}[theorem]{Proposition}
\theoremstyle{definition}

\theoremstyle{definition}

\theoremstyle{definition}

\theoremstyle{definition}
\newtheorem{defn}[theorem]{Definition}
\theoremstyle{definition}

\theoremstyle{definition}

\DeclareTextCompositeCommand{\v}{OT1}{l}{l\nobreak\hspace{-.1em}'}
\title{Sidorenko's conjecture for subdivisions and theta substitutions}
\author{
Seonghyuk Im\thanks{Department of Mathematical Sciences, KAIST, South Korea and Extremal Combinatorics and Probability Group (ECOPRO), Institute for Basic Science (IBS), Daejeon, South Korea.  Supported by the National Research Foundation of Korea (NRF) grant funded by the Korea government(MSIT) No. RS-2023-00210430 and by the Institute for Basic Science IBS-R029-C4. E-mail: {\tt seonghyuk@kaist.ac.kr}}
\and
Ruonan Li\thanks{School of Mathematics and Statistics, Northwestern Polytechnical University, Xi'an, Shaanxi, P.R.~China. Supported by National Natural Science Foundation of China 11901459 and 12131013, China Scholarship Council 202306290113 and the Institute for Basic Science IBS-R029-C4. Email: {\tt rnli@nwpu.edu.cn}}
\and 
Hong Liu\thanks{Extremal Combinatorics and Probability Group (ECOPRO),  Institute for Basic Science (IBS), Daejeon, South Korea. Supported by Institute for Basic Science IBS-R029-C4. Email: {\tt hongliu@ibs.re.kr}}
}
\date{\today}

\begin{document}
\maketitle
\begin{abstract}
    The famous Sidorenko's conjecture asserts that for every bipartite graph $H$, the number of homomorphisms from $H$ to a graph $G$ with given edge density is minimized when $G$ is pseudorandom. We prove that for any graph $H$, a graph obtained from replacing edges of $H$ by generalized theta graphs consisting of even paths satisfies Sidorenko's conjecture, provided a certain divisibility condition on the number of paths. To achieve this, we prove unconditionally that bipartite graphs obtained from replacing each edge of a complete graph with a generalized theta graph satisfy Sidorenko's conjecture, which extends a result of Conlon, Kim, Lee and Lee [J. Lond. Math. Soc., 2018].
\end{abstract}

\section{Introduction}
One fundamental problem in extremal graph theory is to determine the maximum/minimum possible number of $H$ copies in graphs of given edge density. A major conjecture by Sidorenko~\cite{Sidorenko1993, Sidorenko1991} and independently Erd\H{o}s and Simonovits~\cite{Erdos_sidorenkoconj_1984} states that for every bipartite graph $H$, this number is asymptotically minimized by the random graph of the same edge density. Formally, for graphs $H$ and $G$, a \emph{homomorphism} from $H$ to $G$ is a function $f:V(H) \rightarrow V(G)$ such that $f(u)f(v) \in E(G)$ whenever $uv \in E(H)$.
Let $hom(H, G)$ be the number of homomorphisms from $H$ to $G$.
The \emph{homomorphism density} of $H$ in $G$, denoted by $t_H(G)$, is defined by $\frac{hom(H, G)}{v(G)^{v(H)}}$.
If $v(G)$ is sufficiently large compared to $v(H)$, then the homomorphism density is asymptotically the same as the subgraph density. Sidorenko's conjecture is stated as follows.
\begin{conj}[Sidorenko's conjecture]\label{conj:Sidorenko}
    For every bipartite graph $H$ and every graph $G$, we have 
    \begin{equation}\label{eqt:sidorenko}
        t_H(G) \geq t_{K_2}(G)^{e(H)}.
    \end{equation}
\end{conj}
We say a graph $H$ is {\it Sidorenko} if (\ref{eqt:sidorenko}) holds for every graph $G$.
Sidorenko~\cite{Sidorenko1991} confirmed the conjecture for trees, even cycles, complete bipartite graphs, and when one of the parts of the bipartition has size at most $4$.
Despite attracting continuous interests and efforts, and many partial results over recent years (see~\cite{Conlon_approximate_sidorenko_2010,Conlon_Sidorenko_subdivision_2018, Conlon_sidorenko_blowup_2021, Coregliano_induced_sidorenko_2024, Hatami_graph_norms_2010, Kim_sidorenko_product_2016, Li_logarimic_calculus_2011, Szegedy_informationtheoretic_sidorenko_2015}), Sidorenko's conjecture is still far from being completely understood.

Note that the bipartiteness condition in Sidorenko's conjecture is necessary since $G$ is possibly a bipartite graph.
One may ask which condition of $G$ ensures that the (asymptotically) same inequality holds for a nonbipartite graph $H$.
Chung, Graham, and Wilson~\cite{Chung_quasirandom_1989} proved a celebrity result on the characterization of quasirandomness which implies that if for every $U \subseteq V(G)$, the number of edges inside $U$ is $d{|U| \choose 2} + o(v(G)^2)$, then $t_H(G) = d^{e(H)}+o(1)$. 
Kohayakawa, Nagle, R\"{o}dl, and Schacht~\cite{Kohayakawa_KNRS_2010} conjectured that one-side inequality is sufficient to guarantee the lower bound, which is now called KNRS conjecture.
A graph $G$ on $n$-vertex is called $(\varepsilon, d)$-dense if for every $U \subseteq V(G)$ with $|U| \geq \varepsilon n$, there holds $e(G[U]) \geq d {|U| \choose 2}$. Informally, such a graph $G$ is called {\it locally dense}.
\begin{conj}[KNRS conjecture]
    Let $H$ be a graph. Then for every $d, \eta \in (0, 1)$, there exists $\varepsilon = \varepsilon(d, \eta, H)$ such that if $G$ is $(\varepsilon, d)$-dense, then $t_H(G) \geq (1-\eta)d^{e(H)}$.
\end{conj}
We say $H$ is {\it KNRS} if $H$ satisfies the KNRS conjecture. While some partial cases such as complete graphs, complete multipartite graphs~\cite{Kohayakawa_KNRS_2010}, odd cycles~\cite{Reiher_KNRS_oddcycle_2014}, see also~\cite{bradac_counting_2024, kral2024common, Lee_KNRS_2021} for recent developments, the full conjecture is still widely open.

It is clear that if $H$ is Sidorenko, then it is KNRS. Conlon, Kim, Lee, and Lee~\cite{Conlon_Sidorenko_subdivision_2018} showed more connections between these two conjectures: if $H$ is KNRS, then its  $1$-subdivision is Sidorenko.
Furthermore, they proved that if one replaces each edge of a KNRS graph $H$ by $K_{2, t}$, then the resulting graph is Sidorenko. A graph $H$ is called a {\it generalized theta graph} if it is obtained by adding internally disjoint paths between two distinct vertices $u$ and $v$. We call $u$ and $v$ the {\it roots} of $H$.  If all these paths have even number of edges, then we call $H$ an {\it even generalized theta graph}. Note that $K_{2,t}$ is the simplest even generalized theta graph. Conlon, Kim, Lee and Lee~\cite{Conlon_Sidorenko_subdivision_2018} also proved that an even generalized theta graph is Sidorenko. 

Our first result is a common generalization of these two results, providing a new class of Sidorenko graphs. 

\begin{theorem}\label{thm:main_theta_replacement}
    Let $H$ be a graph that satisfies the KNRS conjecture. Let $\Theta$ be an even generalized theta graph with roots $u$ and $v$. 
    If we use $\Theta$ to replace each edge of $H$ by identifying its two vertices with $u$ and $v$ respectively, then the resulting graph is Sidorenko.
    In particular, every even generalized theta graph is Sidorenko.
\end{theorem}
Very recently, Chen, Lin, and Ma~\cite{chen_kohayakawa-nagle-rodl-schacht_2024} independently proved that if $H$ is KNRS, then its $(2k-1)$-subdivision is Sidorenko. 
While some ideas of our proof and Chen, Lin, and Ma's are close: defining auxiliary weighted graphs that count the number of paths between two vertices, and showing that if $G$ is almost regular, then these weighted graphs are locally dense. 
However, to extend paths to generalized theta graphs, we need to show that the Hadamard product of those auxiliary graphs is also locally dense. As the locally dense property is not preserved by the Hadamard product, we need to prove that those auxiliary graphs are correlated.

\medskip

The ``subdivision results'' and the ``$\Theta$ results'' aforementioned are processing a uniform replacement of edges of $H$, and more importantly they all require $H$ to be KNRS. Our next result goes beyond in both aspects, showing that for \emph{any graph $H$}, the graph obtained from  replacing edges of $H$ by \emph{any non-uniform} even generalized theta graphs is Sidorenko provided that the number of paths satisfies a certain divisibility condition.

\begin{theorem}\label{thm:main_non_uniform}
    Let $H$ be an $h$-vertex graph and $H'$ be a graph obtained by replacing each edge $e$ of $H$ with internally vertex-disjoint paths of even lengths connecting the end vertices of $e$. Let $h_e(k)$ denote the number of length $k$ paths in the replacement of $e$. 
    If $\sum_{e \in E(H)} h_e(2k)$ is divisible by ${h \choose 2}$ for all $k \geq 1$, then $H'$ satisfies Sidorenko's conjecture.
    Also, if $\sum_{e \in E(H)} h_e(2k)=0$ for all but one $k$ and $\sum_{e \in E(H)} h_e(2k) \geq {h \choose 2}$ for such exceptional $k$, then $H'$ is Sidorenko.
\end{theorem}

The proof of~\Cref{thm:main_non_uniform} builds on~\Cref{thm:main_theta_replacement} and uses a H\"{o}lder trick inspired by Conlon and Lee~\cite{Conlon_sidorenko_blowup_2021}. Conlon and Lee~\cite{Conlon_sidorenko_blowup_2021} proved that for every bipartite graph $H$, there exists a bipartite graph $H'$ such that their disjoint union $H \cup H'$ is Sidorenko. Using \Cref{thm:main_non_uniform}, we have an analogous statement that if $H$ is a subdivision of a graph, then $H'$ can be taken as a generalized theta graph.
\begin{cor}
    Let $H_0$ be a graph and $H$ be a subdivision of $H_0$ such that every edge is subdivided odd number of times.
    Then there exists a generalized theta graph $\Theta$ such that $H \cup \Theta$ is Sidorenko.
\end{cor}

We remark that a result of a similar spirit for KNRS conjecture was recently proved by Kr{\'a}\v{l}, Volec and Wei \cite{kral2024common}, who showed that for every graph $H$ with girth at least $50$, there exists a larger graph $H’$, which is KNRS and containing $H$ as an induced subgraph.

While \Cref{thm:main_non_uniform} gives a new approach to consider non-uniform replacement, we are not able to deal with subdivisions with different path lengths. 
In the consideration of non-uniform subdivisions, we prove a special case of clique subdivision is Sidorenko. 
\begin{theorem}\label{thm:non_uniform_of_clique}
    Let $h, \ell_1, \ell_2 \geq 1$ be integers and $v \in V(K_h)$ be a vertex.
    Let $H$ be a subdivision of $K_h$ obtained by subdividing each edge $2\ell_1-1$ times if an edge is not incident to $v$ and $\ell_2-1$ times if an edge is incident to $v$.
    Then $H$ is Sidorenko.
\end{theorem}
\Cref{thm:non_uniform_of_clique} is in fact a special case of a more general theorem. To state it, we need a notion of product of two graphs introduced by Brada\v{c}, Sudakov and Wigderson~\cite{bradac_counting_2024}.
For graphs $H_1, H_2$, an independent set $I \subseteq V(H_1)$, and a vertex $a \in V(H_1) \setminus I$, let $H_1 \ltimes_I^a H_2$ be a graph obtained by the following process. 
We start with $|V(H_2)|$ copies of $H_1$ that are identified at $I$ and disjoint otherwise. Let $X$ be the collection of copies of the vertex $a$ from each $H_1$. We add a copy of $H_2$ on $X$ to obtain $H_1 \ltimes_I^a H_2$.
Brada\v{c}, Sudakov and Wigderson~\cite{bradac_counting_2024} recently proved that if $H_1$ and $H_2$ are KNRS, then $H_1 \ltimes_I^a H_2$ is also KNRS.
We note that the product $H_1 \ltimes_I^a H_2$ may not be bipartite even if $H_1$ and $H_2$ are Sidorenko (considering the case that $H_1$ is a path and $H_2$ is an edge). So a direct analog of Brada\v{c}, Sudakov and Wigderson's result is not true for Sidorenko's conjecture.
However, our next result shows that if we subdivide the product (minimally) to make it bipartite, then it becomes Sidorenko.

\begin{theorem}\label{thm:main_semidirect_product}
    Let $H_1$ be a Sidorenko graph and $H_2$ be a KNRS graph.
    Let $I \subseteq V(H_1)$ be an independent set and $a \in V(H_1) \setminus I$ be a vertex and $k \geq 1$ be an integer.
    Let $F$ be a graph obtained by subdividing $2k-1$ times edges of $H_1 \ltimes_I^a H_2$ that corresponds to $H_2$.
    Then $F$ is Sidorenko.
\end{theorem}

We note that our proof provides a stronger statement that instead of using a subdivision of edges corresponding to $H_2$, one can replace those edges with a theta graph of even lengths. 

\begin{proof}[Proof of \Cref{thm:non_uniform_of_clique}]
    We take $H_1$ to be a path of length $\ell_2$ and $H_2 = K_{h-1}$. The set $I$ consists of one of the end vertices of $H_1$ and $a$ is the other end vertex of $H_1$.
    Then subdividing edges of $H_2$ of $H_1 \ltimes_I^a H_2$ by $2\ell_1-1$ times produces the desired graph.
\end{proof}

Recall that even generalized theta graphs are Sidorenko. It is still not known whether a generalized theta graph is KNRS; Brada\v{c}, Sudakov, and Wigderson~\cite{bradac_counting_2024} proved that it satisfies a slightly weaker ``regular KNRS conjecture''. We prove that if we identify two roots of a generalized theta graph, then it is KNRS. We call such graphs flowers. Equivalently, a \emph{flower} is a graph consisting of a collection of cycles sharing a common vertex and disjoint otherwise.

\begin{theorem}\label{thm:flower}
    Any flower is KNRS.
\end{theorem}
By considering the $1$-subdivision of a flower, we see that a flower consisting of even cycles satisfies Sidorenko's conjecture, which recovers a result in~\cite{Conlon_Sidorenko_subdivision_2018}.

\medskip

\noindent\textbf{Organization.}
In Section~\ref{sec:Preliminaries} we introduce the notion of graphon and related lemmas. In Section~\ref{subsec:auxiliary_graph}, we first show that certain auxiliary graphs are locally dense and prove \Cref{thm:main_theta_replacement,thm:main_semidirect_product}. 
We prove \Cref{thm:main_non_uniform} in Section~\ref{subsec:Holder_technique} and
\Cref{thm:flower} in Section~\ref{subsec:proof_for_flowers}.
Concluding remarks are given in Section~\ref{sec:concluding}.

\section{Preliminaries}\label{sec:Preliminaries}
We will work with graphons instead of graphs. A \emph{graphon} is a measurable function $W:[0, 1]^2 \rightarrow [0, 1]$ such that $W(x, y) = W(y, x)$ for (almost) every $x, y \in [0, 1]$. A graphon $W$ is $d$-regular if $\int_{[0, 1]} W(x, y)\mathrm{d}y=d$ for (almost) every $x \in [0, 1]$.
Graphons are considered as a limit object of dense graphs and many statements of graphs have an equivalent form in terms of graphon. 
See~\cite{Lovasz_graph_limits_2012} for the theory of graphon and examples.

For a graph $H$ and a graphon $W$, the homomorphism density of $H$ in $W$ is defined as 
$$t_H(W) = \int_{[0, 1]^{v(H)}} \prod_{ij \in E(H)} W(x_i, x_j) \prod_{i \in [v(H)]} \mathrm{d}x_i.$$
With this notation, Sidorenko's conjecture can be stated in the following equivalent form.
\begin{conj}
    For every bipartite graph $H$ and graphon $W$, we have $t_H(W) \geq t_{K_2}(W)^{e(H)}$.
\end{conj}
In the graph version of Sidorenko's conjecture, one may assume that the host graph $G$ is regular~\cite{Szegedy_entropy_2015}. An analogous statement for the graphon version is also true.

\begin{lemma}[\cite{coregliano_biregularity_2021}, Theorem 8.2]
    A bipartite graph $H$ is Sidorenko if and only if $t_H(W) \geq d^{e(H)}$ for every $d$-regular graphon $W$.
\end{lemma}

We can also consider KNRS conjecture in terms of graphon. We first define a notion that corresponds to locally denseness. Let $\lambda$ be the standard Lebesgue measure on $[0, 1]$.
\begin{defn}
    A graphon $W$ is $d$-locally dense if for every measurable subset $S \subseteq [0, 1]$, the following holds:
    $$\int_{S \times S} W(x, y) \mathrm{d}x\mathrm{d}y \geq d \lambda(S)^2.$$
\end{defn}
The graphon version of KNRS conjecture reads as follows.
\begin{lemma}[\cite{bradac_counting_2024}, Lemma 2.6]
    A graph $H$ is KNRS if and only if $t_H(W) \geq d^{e(H)}$ for every $d$-locally dense graphon $W$.
\end{lemma}

Reiher~\cite{Reiher_KNRS_oddcycle_2014} proved the following lemma which is useful for KNRS conjecture. It says that if a graph is locally dense, then it is also locally dense in a weighted sense.
\begin{lemma}[\cite{Reiher_KNRS_oddcycle_2014}]\label{lem:Reiher_lemma_original}
    Let $G$ be an $n$-vertex $(\varepsilon, d)$-dense graph.
    Let $f:V(G) \rightarrow [0, 1]$ be a function such that $\sum_{v \in V(G)} f(v) \geq \varepsilon n$.
    Then $\sum_{uv \in E(G)} f(u)f(v) \geq d \left(\sum_{v \in V(G)} f(v)\right)^2-n.$
\end{lemma}
We will use the following graphon version.
\begin{lemma}[\cite{bradac_counting_2024}, Lemma 2.8]\label{lem:Reihers_lemma}
    Let $W$ be a $d$-locally dense graphon and $f:[0, 1] \to [0, 1]$ be a measurable function.
    Then, 
    $$\int_{[0, 1]^2} f(x)f(y)W(x, y)\mathrm{d}x\mathrm{d}y \geq d\left(\int_{[0, 1]} f(x)\mathrm{d}x\right)^2.$$
\end{lemma}
We also note the following extension of Reiher's lemma proved by Brada\v{c}, Sudakov, and Wigderson~\cite{bradac_counting_2024}. Roughly speaking, the original Reiher's lemma is for vertex-weighted counting of $K_2$ and it extends to vertex-weighted counting of a KNRS graph $H$. 
\begin{lemma}[\cite{bradac_counting_2024}, Lemma 2.10]\label{lem:Reiher_lemma_for_general_graph}
    Let $H$ be a KNRS graph and $W$ be a $d$-locally dense graphon. 
    Let $f:[0, 1] \to [0, 1]$ be a measurable function.
    Then, 
    $$\int_{[0, 1]^{v(H)}} \prod_{i \in [v(H)]} f(x_i) \prod_{ij \in E(H)} W(x_i, x_j) \prod_{i \in [v(H)]} \mathrm{d}x_i \geq d^{e(H)}\left(\int_{[0, 1]} f(x)\mathrm{d}x\right)^{v(H)}.$$
\end{lemma}

We need the following version of H\"{o}lder's inequality.
\begin{theorem}[H\"{o}lder's inequality]\label{thm:Holder}
    Let $p_1, \ldots, p_k, r \in (0, \infty)$ satisfies $\frac{1}{p_1} + \ldots + \frac{1}{p_k} = \frac{1}{r}$. Then for any measurable functions $f_1, \ldots, f_k:[0, 1] \rightarrow \mathbb{R}$, we have 
     $$\prod_{i=1}^k \left(\int_{[0, 1]}|f_i|^{p_i} \mathrm{d}x\right)^{1/p_i} \geq \left(\int_{[0, 1]} \left(\prod_{i=1}^k |f_i| \right)^{r}\mathrm{d}x\right)^{1/r}.$$
\end{theorem}

\section{Proofs of main results}\label{sec:proofs}

\subsection{Locally dense auxiliary graph}\label{subsec:auxiliary_graph}
For a graphon $W$ and a graph $F$ with vertex set $[f]$ and two (distinct) roots $i,j \in [f]$, we define \emph{$F$-counting kernel} $W^F$ as 
$$W^F(x_i, x_j):= \int_{[0, 1]^{f-2}} \prod_{i_1i_2 \in E(F)} W(x_{i_1}, x_{i_2}) \prod_{k \in [f] \setminus \{i, j\}} \mathrm{d}x_k.$$
We note that if $F$ is not symmetric, then $W^F$ may not be a graphon. 
However, in this paper, we assume that $F$ has an automorphism that swaps the two roots, and therefore $W^F$ is a graphon to avoid some technicalities. 
When $F$ is a path of length $k$ with two roots being its end vertices, we denote $W^F$ by $W^k$. 
Note that $W^k$ is equivalent to the $k$-th matrix power of $W$. 
The following property of $W^F$ is useful.
\begin{lemma}\label{lem:relation_of_aux_and_original}
    Let $F$ be a graph with two roots such that there exists an automorphism that swaps its two roots and $H$ be a graph. 
    Let $H'$ be a graph obtained by replacing each edge $e$ of $H$ with a copy of $F$ such that the roots of $F$ are identified with the end vertices of $e$. Then $t_{H'}(W) = t_H(W^F)$ for every graphon $W$.
\end{lemma}
As it follows from writing the definition of $t_{H'}(W)$ and integrating variables corresponding to the vertices of $V(H') \setminus V(H)$ first, we omit the proof of this lemma. 

The key lemma of this subsection is the following.
\begin{lemma}\label{lem:auxiliary_locally_dense}
    Let $W$ be a $d$-regular graphon.
    Let $\Theta$ be an even generalized theta graph.
    Then $W^{\Theta}$ is $d^{e(\Theta)}$-locally dense.
\end{lemma}

This lemma can be proved by inductively attaching even paths to the roots of an even generalized theta graph. The next lemma captures the essence of this inductive process.

\begin{lemma}\label{lem:aux_attaching_paths}
    Let $W_1$ be a $d_1$-locally dense graphon and $W_2$ be a $d_2$-regular graphon and $k \geq 1$ be an integer.
    Let $W(x, y) = W_1(x, y)W_2^{2k}(x, y)$. 
    Then $W$ is $d_1d_2^{2k}$-locally dense.
\end{lemma}
\begin{proof}
    We first note that as $W_2$ is $d_2$-regular, we have 
    \begin{align*}
        \int_{[0, 1]} W_2^{k}(x, y) \mathrm{d}y & = \int_{[0, 1]^{k}} W_2(x, y_1)W_2(y_1, y_2) \cdots W_2(y_{k-1}, y_k) \prod_{i \in [k]} \mathrm{d}y_i \\
        & = \int_{[0, 1]^{k-1}} W_2(x, y_1)W_2(y_1, y_2) \cdots W_2(y_{k-2}, y_{k-1}) \int_{[0, 1]} W_2(y_{k-1}, y_k) \mathrm{d}y_k \prod_{i \in [k-1]} \mathrm{d}y_i \\
        & = d_2 \int_{[0, 1]^{k-1}} W_2(x, y_1)W_2(y_1, y_2) \cdots W_2(y_{k-2}, y_{k-1}) \prod_{i \in [k-1]} \mathrm{d}y_i.
    \end{align*}
    Therefore, by induction, we have $\int_{[0, 1]} W_2^{k}(x, y) \mathrm{d}y = d_2^k$ for every $k \geq 1$ and almost every $x \in [0, 1]$.
    Let $S \subseteq [0, 1]$ be any measurable set with positive measure.
    Then we have 
    \begin{align*}
        \int_{S \times S} W(x, y) \mathrm{d}x\mathrm{d}y &= \int_{S \times S} W_1(x, y)W_2^{2k}(x, y) \mathrm{d}x\mathrm{d}y \\
        &= \int_{S \times S} W_1(x, y) \int_{[0, 1]} W_2^k(z, x)W_2^k(z, y) \mathrm{d}z\mathrm{d}x\mathrm{d}y \\
        & = \int_{[0, 1]}\int_{S \times S} W_2^k(z, x)W_1(x, y)W_2^k(z, y) \mathrm{d}x\mathrm{d}y\mathrm{d}z \\
        & \geq \int_{[0, 1]} d_1 \left(\int_S W_2^k(z, x)\mathrm{d}x\right)^2 \mathrm{d}z
    \end{align*}
    by applying \Cref{lem:Reihers_lemma} with $f( \cdot ) = W_2^k(z, \cdot) \cdot \mathbf{1}_{S}$.
    By using Jensen's inequality for $x \mapsto x^2$, we have 
    \begin{align*}
    d_1 \int_{[0, 1]} \left(\int_S W_2^k(z, x)\mathrm{d}x\right)^2 \mathrm{d}z & \geq d_1\left(\int_{[0, 1]} \int_S W_2^k(z, x) \mathrm{d}x \mathrm{d}z \right)^2  \\
    & = d_1 \left(\int_{S} \int_{[0, 1]} W_2^k(z, x) \mathrm{d}z \mathrm{d}x \right)^2 \\
    & = d_1 d_2^{2k}\lambda(S)^2.
    \end{align*}
    Therefore, $\int_{S \times S} W(x, y) \mathrm{d}x\mathrm{d}y \geq d_1d_2^{2k}\lambda(S)^2$ for any measurable set $S \subseteq [0, 1]$ which concludes the proof.
\end{proof}

\begin{proof}[Proof of \Cref{lem:auxiliary_locally_dense}]
    We apply the mathematical induction on the number of paths in $\Theta$.
    If $\Theta$ consists of a single path of length $2k$, then we apply \Cref{lem:aux_attaching_paths} with $W_1 \equiv 1$ and $W_2=W$. 
    Then $W^{\Theta} = W^k$ is $d^{2k}$-locally dense by \Cref{lem:aux_attaching_paths}.
    
    If $\Theta$ has more than one path, then let $\Theta'$ be a theta graph obtained by deleting one path of length $2k$ from $\Theta$.
    By the induction hypothesis, $W^{\Theta'}$ is $d^{e(\Theta')}$-locally dense.
    We apply \Cref{lem:aux_attaching_paths} with $W_1=W^{\Theta'}$ and $W_2=W$. Then the Hadamard product of $W^{\Theta'}$ and $W^{2k}$ is $d^{e(\Theta')}d^{2k}=d^{e(\Theta)}$-locally dense by \Cref{lem:aux_attaching_paths}. As the Hadamard product of $W^{\Theta'}$ and $W^{2k}$ is $W^{\Theta}$, it concludes the proof.
\end{proof}

We are now ready to prove \Cref{thm:main_theta_replacement}.
\begin{proof}[Proof of \Cref{thm:main_theta_replacement}]
    Let $H$ be a KNRS graph and $\Theta$ be an even generalized theta graph.
    Let $H'$ be a graph obtained by replacing each edge of $H$ by $\Theta$.
    Let $W$ be a $d$-regular graphon.
    Then by \Cref{lem:auxiliary_locally_dense}, $W^{\Theta}$ is $d^{e(\Theta)}$-locally dense. Therefore, $t_{H'}(W) = t_{H}(W^{\Theta}) \geq d^{e(\Theta)e(H)} = d^{e(H')}$ as $H$ is KNRS.
    Therefore, $H'$ is Sidorenko.
\end{proof}
We now prove \Cref{thm:main_semidirect_product}.
\begin{proof}[Proof of \Cref{thm:main_semidirect_product}]
    Let $W$ be a $d$-regular graphon.
    We label the vertices of $H_1$ by $1, 2, \ldots, h$ where $I=\{1, 2, \ldots, t\}$ and $a=h$.
    For each $x_1, \ldots, x_t \in [0, 1]$, we define the normalized number of embeddings of $H_1$ (with $1, 2, \ldots, t, $ and $h$ being fixed) by   
    $$f_{x_1, \ldots, x_t}^W(x_h) = \int_{[0, 1]^{h-t-1}} \prod_{ij \in E(H_1)} W(x_i, x_j) \prod_{t+1 \leq i \leq h-1} \mathrm{d}x_i.$$    
    Then we observe that 
    $$t_{F}(W) = \int_{[0, 1]^{t}} \left( \int_{[0, 1]^{v(H_2)}} \prod_{i \in [v(H_2)]} f_{x_1, \ldots, x_t}^W(y_i) \prod_{ij \in E(H_2)} W^{2k}(y_i, y_j) \prod_{i \in [v(H_2)]} \mathrm{d}y_i\right) \prod_{i \in [t]} \mathrm{d}x_i.$$
    As $W^{2k}$ is $d^{2k}$-locally dense by \Cref{lem:auxiliary_locally_dense} and $H_2$ is KNRS, we can apply \Cref{lem:Reiher_lemma_for_general_graph} to obtain 
    $$t_{F}(W) \geq \int_{[0, 1]^{t}} \left( \int_{[0, 1]} f_{x_1, \ldots, x_t}^W(z) \mathrm{d}z \right)^{v(H_2)} d^{2ke(H_2)} \prod_{i \in [t]} \mathrm{d}x_i.$$
    By Jensen's inequality for $x \mapsto x^{v(H_2)}$ and that $H_1$ is Sidorenko, we have 
    \begin{align*}
        t_{F}(W) &\geq d^{2ke(H_2)} \int_{[0, 1]^{t}} \left( \int_{[0, 1]} f_{x_1, \ldots, x_t}^W(z) \mathrm{d}z \right)^{v(H_2)}  \prod_{i \in [t]} \mathrm{d}x_i \\
        & \geq d^{2ke(H_2)} \left(\int_{[0, 1]^{t}} \int_{[0, 1]} f_{x_1, \ldots, x_t}^W(z) \mathrm{d}z \prod_{i \in [t]} \mathrm{d}x_i \right)^{v(H_2)} \\
        & \geq d^{2ke(H_2) + v(H_2)e(H_1)}.
    \end{align*}    
    Therefore, $F$ is Sidorenko.
\end{proof}

\subsection{Uniformization via H\"{o}lder's inequality}\label{subsec:Holder_technique}
This section's main lemma is the following, which allows us to reduce a non-uniform replacement case into a uniform case.
\begin{lemma}\label{lem:uniformization_via_Holder}
    Let $H$ be a graph with $V(H)=[h]$ and $H'$ be a graph obtained by replacing each edge $e$ of $H$ with internally vertex-disjoint paths of even lengths connecting the end vertices of $e$. Let $h_e(k)$ denote the number of length $k$ paths in the replacement of $e$ and let $\alpha_k = \sum_{e \in E(H)} h_e(k)/{h \choose 2}$.
    Then we have 
    $$t_{H'}(W) \geq \int_{[0, 1]^h}\prod_{k \geq 1} \prod_{ij \in {V(H) \choose 2}} W^k(x_i, x_j)^{\alpha_k} \prod_{i \in [h]} \mathrm{d}x_i.$$
\end{lemma}
\begin{proof}
    By integrating variables corresponding to $V(H') \setminus V(H)$ first, we have
    $$ t_{H'}(W) = \int_{[0, 1]^h}\prod_{k \geq 1} \prod_{ij \in {V(H) \choose 2}} W^k(x_{i}, x_{j})^{h_{ij}(k)} \prod_{i \in [h]} \mathrm{d}x_i,$$
    where we set $h_{ij}(k)=0$ when $ij$ is not an edge of $H$.
    We now observe that for any permutation $\varphi: V(H) \rightarrow V(H)$, there holds
    $$ t_{H'}(W) = \int_{[0, 1]^h}\prod_{k \geq 1} \prod_{ij \in {V(H) \choose 2}} W^k(x_{\varphi(i)}, x_{\varphi(j)})^{h_{ij}(k)} \prod_{i \in [h]} \mathrm{d}x_i,$$
    as it is equivalent to the relabeling of vertices of $H$.
    By taking product over all permutations $\varphi:[h] \rightarrow [h]$ and applying H\"{o}lder's inequality(\Cref{thm:Holder}) with $p_k=1$ for all $k \in [h!]$ and $r=1/h!$, we have 
    \begin{align*}
        t_{H'}(W)^{h!} & = \prod_{\varphi:[h] \rightarrow [h]} \left(\int_{[0, 1]^h}\prod_{k \geq 1} \prod_{ij \in {V(H) \choose 2}} W^k(x_{\varphi(i)}, x_{\varphi(j)})^{h_{ij}(k)} \prod_{i \in [h]} \mathrm{d}x_i \right) \\
        & \geq \left(\int_{[0, 1]^h} \prod_{\varphi:[h] \rightarrow [h]}\prod_{uv \in E(H)} \prod_{k \geq 1} W^k(x_{\varphi(u)}, x_{\varphi(v)})^{h_{uv}(k)/h!}  \prod_{i \in [h]} \mathrm{d}x_i \right)^{h!}.
    \end{align*}
    Now fix distinct $i, j \in V(H)$ and $k\in\mathbb{N}$. Consider the power of $W^k(x_i, x_j)$ in the product $$\prod_{\varphi:[h] \rightarrow [h]}\prod_{uv \in E(H)}  W^k(x_{\varphi(u)}, x_{\varphi(v)})^{h_{uv}(k)/h!}.$$
    For any distinct $i', j' \in V(H)$, the number of permutations $\varphi$ that maps $i'$ to $i$ and $j'$ to $j$ is $(h-2)!$.
    As $W^k(x_i, x_j) = W^k(x_j, x_i)$, we have $2(h-2)!$ permutations such that $W^k(x_i, x_j) = W^k(x_{\varphi(i')}, x_{\varphi(j')})$.
    Therefore, the power of $W^k(x_i, x_j)$ in $\prod_{\varphi:[h] \rightarrow [h]}\prod_{uv \in E(H)} W^k(x_{\varphi(u)}, x_{\varphi(v)})^{h_{uv}(k)/h!}$ is $$\sum_{u'v' \in {V(H) \choose 2}} h_{u'v'}(k) 2(h-2)!/h! = \alpha_k.$$
    Hence we have 
    $$t_{H'}(W)^{h!} \geq \left(\int_{[0, 1]^h}\prod_{k \geq 1} \prod_{ij \in {V(H) \choose 2}} W^k(x_i, x_j)^{\alpha_k} \prod_{i \in [h]} \mathrm{d}x_i\right)^{h!},$$ 
    which completes the proof.
\end{proof}
We are now ready to prove \Cref{thm:main_non_uniform}.
\begin{proof}[Proof of \Cref{thm:main_non_uniform}]
    Let $H, H'$ be as in the statement and $W$ be a $d$-regular graphon.
    Let $\alpha_k = \sum_{e \in E(H)} h_e(k)/{h \choose 2}$. 
    Note that by the assumption, $\alpha_{2k-1}=0$ for every $k \geq 1$.
    By \Cref{lem:uniformization_via_Holder}, we have 
    \begin{align}
        t_{H'}(W) \geq \int_{[0, 1]^h}\prod_{k \geq 1} \prod_{ij \in {V(H) \choose 2}} W^{2k}(x_i, x_j)^{\alpha_{2k}} \prod_{i \in [h]} \mathrm{d}x_i. \label{eq:uniformization}
    \end{align}
    For the first case, if all of $\sum_{e \in E(H)} h_e(2k)$ is divisible by ${h \choose 2}$, then $\alpha_{2k}$ is an integer for all $k \geq 1$.
    Let $\Theta$ be a theta graph which consists of $\alpha_{2k}$ paths of length $2k$ for all $k\ge 1$ and let $F$ be a graph obtained by replacing each edge of $K_h$ by $\Theta$.
    Then by \Cref{thm:main_theta_replacement}, $F$ is Sidorenko. 
    Noting that $e(F)={h\choose 2}\sum_{k} \alpha_{2k}\cdot 2k=\sum_{k}\sum_{e \in E(H)} h_e(2k)=e(H')$, we then have 
    $$t_{H'}(W) \geq \int_{[0, 1]^h}\prod_{k \geq 1} \prod_{ij \in {V(H) \choose 2}} W^{2k}(x_i, x_j)^{\alpha_{2k}} \prod_{i \in [h]} \mathrm{d}x_i = t_F(W) \geq d^{e(F)} = d^{e(H')}.$$
    Hence $H'$ is Sidorenko.

    We now consider the second case. If $\sum_{e \in E(H)} h_e(2k)=0$ for all but one $k$ and $\sum_{e \in E(H)} h_e(2k) \geq \binom{h}{2}$ for such exceptional $k$, then $\alpha_{2k} \geq 1$ for such exceptional $k$ and $\alpha_{2k}=0$ for all other $k$.
    Then \eqref{eq:uniformization} becomes 
    $$t_{H'}(W) \geq \int_{[0, 1]^h} \left(\prod_{ij \in {V(H) \choose 2}} W^{2k}(x_i, x_j)\right)^{\alpha_{2k}} \prod_{i \in [h]} \mathrm{d}x_i$$
    for some $k \geq 1$.
    By applying Jensen's inequality with $x \mapsto x^{\alpha_k}$, we have 
    \begin{align*}
        \int_{[0, 1]^h} \left(\prod_{ij \in {V(H) \choose 2}} W^{2k}(x_i, x_j)\right)^{\alpha_{2k}} \prod_{i \in [h]} \mathrm{d}x_i &\geq \left(\int_{[0, 1]^h} \prod_{ij \in {V(H) \choose 2}} W^{2k}(x_i, x_j) \prod_{i \in [h]} \mathrm{d}x_i\right)^{\alpha_{2k}}\\
        &= t_{K_h}(W^{2k})^{\alpha_{2k}} \geq d^{e(H')},
    \end{align*}
    where the last inequality comes from the fact that $(2k-1)$-subdivision of $K_h$ is Sidorenko.
    Therefore, $H'$ is Sidorenko.
\end{proof}

\subsection{Flowers are KNRS}\label{subsec:proof_for_flowers}
In this section, we prove \Cref{thm:flower}.
\begin{proof}[Proof of \Cref{thm:flower}]
    Let $W$ be a $d$-locally dense graphon and let $H$ be a flower.
    Let $C_1, \ldots, C_{r+s}$ be cycles of $H$ and $v$ be the vertex that all cycles intersect where $C_1, \ldots, C_r$ has length $2\ell_1+1, \ldots, 2\ell_r+1$ and $C_{r+1}, \ldots, C_{r+s}$ has length $2\ell_{r+1}, \ldots, 2\ell_{r+s}$.
    Then we have 
    \begin{align*}
    t_H(W) &= \int_{[0, 1]}\left(\int_{[0, 1]^{2r}} \prod_{i \in [r]} W^{\ell_i}(x, y_i)W^{\ell_i}(x, y'_i) W(y_i, y'_i) \prod_{i \in [r]} \mathrm{d}y_i \mathrm{d}y'_i \right)\\
    &~~~~~~~~~~~\cdot \left(\int_{[0, 1]^{s}} \prod_{r+1 \leq i \leq r+s} W^{\ell_i}(x, z_i)^2 \prod_{r+1 \leq i \leq r+s} \mathrm{d}z_i \right) \mathrm{d}x.
    \end{align*}
    By \Cref{lem:Reihers_lemma}, we have 
    $$\int_{[0, 1]^{2}} W^{\ell_i}(x, y_i)W^{\ell_i}(x, y'_i) W(y_i, y'_i) \mathrm{d}y_i\mathrm{d}y'_i \geq d \left(\int_{[0, 1]} W^{\ell_i}(x, y) \mathrm{d}y \right)^2$$
    for every $x \in [0, 1]$ and $i \in [r]$.
    Also, by Jensen's inequality, $$\int_{[0, 1]^{s}} \prod_{r+1 \leq i \leq r+s} W^{\ell_i}(x, z_i)^2 \prod_{r+1 \leq i \leq r+s} \mathrm{d}z_i \geq \prod_{r+1 \leq i \leq r+s} \left(\int_{[0, 1]} W^{\ell_i}(x, y) \mathrm{d}y \right)^2.$$ 
    Thus, using Jensen's inequality again we have
    \begin{align*}t_H(W) & \geq d^r \int_{[0, 1]} \prod_{i \in [r]} \left(\int_{[0, 1]} W^{\ell_i}(x, y) \mathrm{d}y \right)^2 \prod_{r+1 \leq i \leq r+s} \left(\int_{[0, 1]} W^{\ell_i}(x, y) \mathrm{d}y \right)^2 \mathrm{d}x \\
    & \geq d^r \left(\int_{[0, 1]} \prod_{i \in [r+s]} \left(\int_{[0, 1]} W^{\ell_i}(x, y)\mathrm{d}y \right) \mathrm{d}x \right)^2  \\
    & = d^r \left(\int_{[0, 1]} \int_{[0, 1]^{r+s}} \prod_{i \in [r+s]}W^{\ell_i}(x, y_i) \prod_{i \in [r+s]}\mathrm{d}y_i \mathrm{d}x \right)^2.
    \end{align*}
    Let $F$ be a tree obtained by attaching $r+s$ paths of length $\ell_1, \ldots, \ell_{r+s}$ to a single vertex.
    Then we observe that $t_F(W) = \int_{[0, 1]} \int_{[0, 1]^{r+s}} \prod_{i \in [r+s]}W^{\ell_i}(x, y_i) \prod_{i \in [r+s]}\mathrm{d}y_i \mathrm{d}x$. Also, $F$ is Sidorenko as it is a tree. 
    Therefore, 
    $$t_H(W) \geq d^r t_F(W)^2 \geq d^r d^{2e(F)} = d^{r+\sum_{i \in [r+s]} 2\ell_i} = d^{e(H)},$$ 
    so $H$ is KNRS as desired.
\end{proof}

\section{Concluding remarks}\label{sec:concluding}
In this paper, we provide a large class of subdivisions that are Sidorenko. 
We believe that our results can be extended to a more general setting.
\begin{conj}\label{conj:sidorenko_subdivision}
    Let $H$ be a Sidorenko graph and $H'$ is a subdivision of $H$.
    Then $H'$ is Sidorenko whenever it is bipartite.
\end{conj}
We remark that the uniform case of this conjecture can be proved easily.
\begin{prop}
    Let $H$ be a Sidorenko graph and $H'$ be an $\ell$-subdivision of $H$.
    Then $H'$ is Sidorenko.
\end{prop}
\begin{proof}
    Let $W$ be a $d$-regular graphon.
    Then we obtain 
    \begin{align*}
        t_{H'}(W)  = \int_{[0, 1]^{v(H)}} \prod_{ij \in E(H)} W^{\ell+1}(x_i, x_j) \prod_{i \in [v(H)]} \mathrm{d}x_i  \geq t_{K_2}(W^{\ell+1})^{e(H)} \geq d^{e(H')},
    \end{align*}
    where the last inequality comes from the fact that the path of length $\ell+1$ is Sidorenko.
\end{proof}

Another special case of \Cref{conj:sidorenko_subdivision} is to consider {\it odd generalized theta graphs}. We say a simple graph $H$ is an odd generalized theta graph if it can be obtained by adding internally disjoint odd paths between two fixed root vertices $u$ and $v$. The Sidorenko property of odd generalized theta graphs can be derived from 
the result by Li and Szegedy~\cite{Li_logarimic_calculus_2011} or from the strong tree-decomposition result by Conlon-Kim-Lee-Lee~\cite{Conlon_Sidorenko_subdivision_2018}. 

\begin{prop}\label{prop:odd-theta}
    Every odd generalized theta graph is Sidorenko.
\end{prop}
\begin{proof}[Proof Sketch]
    For $k\geq 2$ and odd integers $\ell_1\geq \ell_2\geq \cdots\geq \ell_k\geq 1$, let $T_0$ be a tree obtained by attaching $k$ paths of lengths $\ell_k, \frac{\ell_1-\ell_k}{2}, \frac{\ell_2-\ell_k}{2},\ldots,\frac{\ell_{k-1}-\ell_k}{2}$ to a vertex $s$. Assume the other ends of the paths are $t,x_1,x_2,\ldots,x_{k-1}$, respectively. Particularly, if $\ell_i=\ell_k$ for some $i\in [1,k-1]$, then $x_i=s$. Let $T_i$ be a path of length $\frac{\ell_i+\ell_k}{2}$ connecting $t$ and $x_i$ such that the internal vertices of $T_i$ is disjoint with $\cup_{j=0}^{i-1} V(T_j)$. Let $\Theta=\cup_{i=0}^{k-1} T_i$. Obviously, $\Theta$ is an odd generalized theta graph for which those paths connecting two roots are of lengths $\ell_1,\ell_2,\ldots,\ell_k$, respectively. Let $\mathcal{F}=\{T_0,T_1,\ldots,T_{k-1}\}$ and let $\mathcal{T}$ be a $(k-1)$-edge star on $\mathcal{F}$ with $T_0$ being the $(k-1)$-degree vertex. Then $(\mathcal{F},\mathcal{T})$ is a tree decomposition of $\Theta$. Such a decomposition indicates that $\Theta$ is strongly tree-decomposable and therefore is Sidorenko according to Theorem 1.2 in the paper of Conlon-Kim-Lee-Lee~\cite{Conlon_Sidorenko_subdivision_2018}.
\end{proof}

\section*{Acknowledgements}
We would like to thank Joonkyung Lee for fruitful discussions and thank Jon Noel for telling us about \Cref{prop:odd-theta}.

\bibliographystyle{abbrv}
\bibliography{references}
\end{document}